\documentclass[10pt, reqno]{amsart}
\pdfoutput=1

\usepackage{amsthm, amssymb, amsmath}
 \usepackage{fullpage}
\usepackage{verbatim}
\usepackage{calrsfs}
\usepackage{graphicx}
\usepackage{epsfig}
\usepackage{color}
\usepackage[all]{xy}

\newtheorem{thm}{Theorem}[section]
\newtheorem{lem}[thm]{Lemma}
\newtheorem{cor}[thm]{Corollary}
\newtheorem{prop}[thm]{Proposition}

\newtheorem*{conjecture*}{Conjecture}

\newtheorem*{Mordell}{Strong Mordell Conjecture}

\theoremstyle{remark} 
\newtheorem*{question*}{Question}

\newtheorem{remark}[thm]{Remark}

\theoremstyle{definition}

\newcommand{\ZZ}{\mathbb{Z}}     
\newcommand{\RR}{\mathbb{R}}     
\newcommand{\PP}{\mathbb{P}}      
\newcommand{\Aff}{\mathbb{A}}      
\newcommand{\QQ}{\mathbb{Q}}      

\newcommand{\be}{\begin{equation}}
\newcommand{\ee}{\end{equation}}
\newcommand{\benn}{\begin{equation*}}
\newcommand{\eenn}{\end{equation*}}
\newcommand{\ba}{\begin{aligned}}
\newcommand{\ea}{\end{aligned}}
\newcommand{\bbm}{\begin{bmatrix}}
\newcommand{\ebm}{\end{bmatrix}}
\newcommand{\bpm}{\begin{pmatrix}}
\newcommand{\epm}{\end{pmatrix}}
\newcommand{\bi}{\begin{itemize}}
\newcommand{\ei}{\end{itemize}}
 
\newcommand{\xhelp}[1]{\textbf{Xander: #1}}

\newcommand{\Spec}{\operatorname{Spec}}

\newcommand{\supp}{\operatorname{supp}}   

\newcommand{\alg}[1]{\overline{#1}}   

\newcommand{\pc}[1]{Y^{\operatorname{pre}}\left(#1\right)}

\title[Effective Mordell and pre-images]{A remark on the Effective Mordell Conjecture and 
rational pre-images under quadratic dynamical systems}
\author{X.W.C. Faber}
\address{
Department of Mathematics and Statistics \\
McGill University \\
Montr\'eal, Qc  H3A 2K6 \\ 
CANADA} 
\email{xander@math.mcgill.ca}
\urladdr{http://www.math.mcgill.ca/xander/}

 

\begin{document}

\begin{abstract}
		Fix a rational basepoint $b$ and a rational number $c$. For the quadratic dynamical system 
		$f_c(x) = x^2+c$, it has been shown that the number of rational points in the backward orbit of $b$ is 
		bounded independent of the choice of rational parameter $c$. In this short note we investigate the 
		dependence of the bound on the basepoint $b$, assuming a strong form of the Mordell Conjecture. 

\vskip 0.5\baselineskip

\noindent  {\bf R\textsc{\'esum\'e}. 
Une remarque sur la Conjecture de Mordell Effective et les pr\'e-images rationnelles par les syst\`emes dynamiques quadratiques.} 
Soit $b$ un point cible rationnel, et soit $c$ un nombre rationnel. Pour le syst\`eme dynamique quadratique $f_c (x) = x ^2 + c$, il a \'et\'e montr\'e que le nombre de points rationnels dans l'orbite inverse de $ b$ est born\'ee 
ind\'ependamment du choix de param\`etre rationnel $c$. Dans cette courte note, nous \'etudions la d\'ependance en le point cible $b$, en supposant une forme forte de la Conjecture de Mordell.

\end{abstract}

\maketitle

\section{Introduction}

	Let $k$ be a number field. For any $c \in k$, define a morphism $f_c: \Aff^1(k) \to \Aff^1(k)$ by $f_c(x) = x^2 + c$. For an integer $N \geq 1$, write $f_c^N$ for the $N$-fold composition of $f_c$. If we fix a ``target'' $b \in \Aff^1(k)$, we can ask questions about the size of the set of \textbf{$k$-rational iterated pre-images of $b$}:
	\be 
	\label{Eq: Pre-images}
		\bigcup_{N \geq 1} f^{-N}_c(b)(k) = \{x \in \Aff^1(k): f_c^N(x) = b \text{ for some $N \geq 1$}\}.
	\ee
These are the $k$-valued points that eventually map to $b$. For fixed $b, c$, this set is finite by Northcott's theorem. In \cite{FHIJMTZ} it was proved that for fixed $b \in \Aff^1(k)$, the cardinality of the set of $k$-rational iterated pre-images can be bounded independently of the parameter $c$:

\begin{thm}[{\cite[Theorem~4.1]{FHIJMTZ}}] 
\label{Thm: FHIJMTZ}
Let $k$ be a number field. For each $b \in \Aff^1(k)$,  
		\benn
			\sup_{c \in k} \# \Bigg\{ \bigcup_{N \geq 1} f^{-N}_c(b)(k) \Bigg\} < \infty.
		\eenn
			
\end{thm}

A natural question is, ``How does this quantity depend on the target $b$?'' The proof of Theorem~\ref{Thm: FHIJMTZ} uses Faltings' Theorem (n\'ee Mordell Conjecture) 
to bound the number of $c \in k$ for which $b$ has an $N^{\mathrm{th}}$ pre-image for large $N$. But without a height bound for these values of $c$, we cannot hope to bound the set in \eqref{Eq: Pre-images}.

	With these considerations in mind, we apply a strong form of the Mordell Conjecture for a certain family of curves to see how these quantities depend on the target $b$. A precise statement of the conjecture is given in Section~\ref{Sec: Mordell}. For the following statement, we write $h: \Aff^1(\alg{\QQ}) \to \RR$ for the absolute Weil height.
	
\begin{thm}
\label{Thm: Height Bound for c}
	Assume the Strong Mordell Conjecture in the form given in \S2. Let $k$ be a number field. There exist constants $\alpha_1, \alpha_2 > 0$, depending only on $k$, with the following property: if $b \in \Aff^1(k)$ and $c \in k$ are such that $f_c^{-N}(b)(k)$ is nonempty for some $N \geq 5$, then 
	\[
		h(c) \leq \alpha_1 h(b) + \alpha_2.
	\]
\end{thm}

	An estimate of the form $h(c) = O(h(b))$ is the best one can hope for. Indeed, let $c \in \ZZ$ be large and set $b = f_c^5(0)$. Since $|b| \approx |c|^{16}$, it follows that
	\[
		h(b) = \log |b| \approx 16 \log|c| = 16h(c).
	\]

\begin{cor}
\label{Cor: Bound for pre-images}
	Fix a number field $k$. Assuming the Strong Mordell Conjecture in the form given in~\S2, there exists an effectively computable constant $\gamma = \gamma(k)$ such that for any $b \in \Aff^1(k)$, we have
	\[
		\sup_{c \in k} \# \Bigg\{\bigcup_{N \geq 1} f_c^{-N}(b)(k) \Bigg\} \ll_k \, \exp\left(\gamma h(b)\right).
	\]
\end{cor}


The remainder of this article is organized as follows. In \S\ref{Sec: Mordell} we state the Strong Mordell Conjecture and discuss an alternative form that is better for applications. The observations in this section are presumably well-known. 
We apply these considerations to a particular family of curves of genus~17 in order to deduce Theorem~\ref{Thm: Height Bound for c} in \S\ref{Sec: Theorem}. Finally, we discuss the canonical height for the morphism $f_c$ and use height comparison estimates to prove the corollary in \S\ref{Sec: Corollary}.   \newline

\noindent \textbf{Acknowledgments:} My thanks go to Patrick Ingram for an invaluable discussion on canonical heights in families,  to Jordan Ellenberg for pointing out an improvement in the statement of Corollary~\ref{Cor: Bound for pre-images}, and to the anonymous referee for catching many small mistakes. This work was supported in part by a National Science Foundation Postdoctoral Research Fellowship.




\section{A Strong Mordell Conjecture}
\label{Sec: Mordell}

\noindent  \textbf{Conventions and Notation:} All varieties and morphisms are defined over a fixed number field~$k$. We set the following notation for this section:\newline

	\begin{tabular}{l c l}
		$T$ & \hspace{0.2cm} &  smooth connected projective variety \\
		$U \subseteq T$ & & nonempty Zariski open subset of $T$ \\
		$\xi$ & & generic point of $T$  (or $U$)\\
		$ \pi: X \to U$ & & smooth family of projective geometrically \\
			& &  \hspace{0.2cm}connected curves of genus $g \geq 2$ \\
		$X_t = \pi^{-1}(t)$ & & fiber of $\pi$ over the point $t \in U$ \\
		$\sigma: U \to X$ & & section of $\pi$ \\
		$D_T$ & & ample divisor on $T$ \\
		$h_{T, D_T}$ & & fixed Weil height associated to $D_T$ \\
		$j_t:X_t \to \mathrm{Jac}(X_t)$ && closed immersion satisfying  \\
			& & \hspace{0.2cm}$j_t(x) = [x] - [\sigma(t)]$ for $x \in X_t\left(\alg{k} \right)$ \\
		$\hat{h}_{t}$
			& & N\'eron-Tate canonical height on the fiber $X_t$ given   \\
			& & \hspace{0.2cm}by pullback of the canonical height on $\mathrm{Jac}(X_t)$ \\
			& &\hspace{0.2cm}associated to the symmetrized theta divisor $\Theta_t + [-1]^*\Theta_t$, \\
			& & \hspace{0.2cm}where 
				$\Theta_t = \underbrace{j_t(X_t) + \cdots + j_t(X_t)}_{g-1 \text{ summands}}$.
	\end{tabular} \newline

	The following effective form of the Mordell conjecture appears in \cite[Conjecture F.4.3.2]{Hindry_Silverman_Book_2000}\footnote{Because of the ambiguity in fixing a Weil height on $T$, the statement as given there cannot be true unless one replaces $ch_T(t)$ by $c_1h_T(t) + c_2$.}:

\begin{Mordell}[Version 1]
	There are effectively computable constants $A, B > 0$ --- depending on the number field $k$, the family $\pi: X \to U$, the height function $h_{T, D_T}$, and the section $\sigma$ --- such that the following bound holds:
	\be
	\label{Eq: Mordell Bound}
		\hat{h}_{t}(x) \leq A \;h_{T, D_T}(t) + B  \quad \text{for all $t \in U(k),  x \in X_t(k)$.}
	\ee
\end{Mordell}

	The best result toward the conjecture is due to de Diego \cite{deDiego_Faltings_in_Families_1996}, in which she extends to families the proof of the Mordell Conjecture as given by Vojta and simplified by Bombieri, but initially proved by Faltings~\cite{Vojta_Siegel_Mordell_1991, Bombieri_Mordell_1990, Faltings_Mordell}. Her result asserts that if $t \in U(k)$ and $x \in X_t(k)$, then either the height estimate~\eqref{Eq: Mordell Bound} holds, or else $x$ lies in an exceptional set of cardinality bounded by a function of the Mordell-Weil rank of the Jacobian $\mathrm{Jac}(X_t)$. (The constants in de Diego's result are independent of the field extension $k$, while we cannot expect such uniformity in the above conjecture.)

	The Strong Mordell Conjecture is perhaps most naturally stated using the canonical height on the fibers of the family. But if one is presented with a family given by explicit equations, then it would be more useful in practice to take advantage of the given system of coordinates. 
		
\begin{Mordell}[Version 2]
	Let $S$ be a smooth connected projective variety and $D_S$ an ample divisor on $S$. Fix a Weil height $h_{S, D_S}$ associated to $D_S$. Suppose that $\gamma: X \to S$ is a morphism such that $\gamma^*D_S$ is relatively ample for $\pi$. Then there are effectively computable constants $A, B > 0$ --- depending on the number field $k$, the family $\pi: X \to U$, the morphism $\gamma: X \to S$, and the height functions $h_{T, D_T}$ and $h_{S, D_S}$ --- such that the following bound holds:
	\[
		h_{S, D_S}(\gamma(x)) \leq A \;h_{T, D_T}(t) + B  \qquad \text{for all } t \in U(k) \text{ and } x \in X_t(k).
	\]
\end{Mordell}

	These two versions of the Strong Mordell Conjecture are equivalent. Indeed, an argument similar to the proof of Theorem~3.1 of \cite{Call_Silverman_1993} shows that there exist positive constants $\alpha_1, \alpha_2, \alpha_3$ such that
	\benn
			\frac{1}{\alpha_1}\hat{h}_{ t}(x) - \alpha_2 h_{T, D_T}(t) - \alpha_3 
			\leq h_{S, D_S}\left(\gamma(x)\right) 
			\leq \alpha_1\hat{h}_{ t}(x) + \alpha_2 h_{T, D_T}(t) + \alpha_3,
	\eenn
 valid for all $t \in U\left(\alg{k}\right)$ and $x \in X_t\left(\alg{k}\right)$. 

\begin{remark}
	An effective form of the usual Mordell Conjecture follows from the Strong Mordell Conjecture. For if $C/k$ is a smooth projective curve of genus at least~2, then we consider the family $X = C \times \PP^1 \to \PP^1$ given by projection on the second factor. Let $h$ denote the absolute logarithmic height on $\PP^1$. Choosing any point $t_0 \in \PP^1(k)$, the Strong Mordell Conjecture gives
	\[
		x \in C\left(k\right) \Longrightarrow \hat{h}(x) = \hat{h}_{t_0}(x) \leq A \; h(t_0) + B.
	\]
That is, $C\left(k\right)$ is a set of bounded height, and hence finite by Northcott's theorem. 
\end{remark}

\begin{remark}
\label{Rem: Singular Fibers}
	For Version~2 of the Strong Mordell Conjecture, there is no harm in supposing that $\pi: X \to U$ is only flat, generically smooth, and that every irreducible component of a fiber $X_t$ has geometric genus at least~2. This allows for finitely many singular fibers (and fibers with multiple irreducible components). Indeed, we could apply Version~2 of the conjecture to the family given by deleting the singular fibers, and then adjust the constants after applying the previous remark to the normalization of each singular fiber.
\end{remark}

\section{The Family of $5^{\mathrm{th}}$ Pre-Image Curves}
\label{Sec: Theorem}

	Recall that $f_c(x) = x^2 + c$. We consider the hypersurface $Y \subset \Aff^3 = \Spec \QQ[c,x,b]$ defined by the equation $f_c^5(x) = b$. We view it as a family of curves via the morphism $\pi_0: Y \to \Aff^1$ satisfying $\pi_0(c,x,b) = b$. 	For fixed $b \in \Aff^1\left(\alg{\QQ}\right)$, the points of $Y_b(\alg{\QQ}) = \pi_0^{-1}(b)$ parametrize pairs $(c,x)$ such that $x$ is a $5^{\mathrm{th}}$ pre-image of $b$ under the map $f_c$.  This family is generically smooth by \cite[Prop~2.1]{FHIJMTZ}. Each of the smooth fibers has genus~17 \cite[Thm.~3.2]{ FHIJMTZ}, and the computations summarized in Table~3.4 of \textit{loc. cit.} show that even if a fiber $Y_b$ is singular, all of its irreducible components have genus at least~3. 
		
	In order to apply the Strong Mordell Conjecture to this family, we must first compactify its fibers. One approach is to consider the closed immersion $\psi: Y \to \Aff^5 \times \Aff^1$ defined by
	\[(c,x,b) \mapsto \left( \left(x, f_c(x), f_c^2(x), f_c^3(x), f_c^4(x)\right), b \right).\]
Write $X = \alg{Y} \subset \PP^5 \times \Aff^1$ for the relative projective closure. An explicit ideal of equations can be written down for $X$, and then one verifies by the Jacobian criterion that the family $\pi: X \to \Aff^1$ given by the second projection is generically smooth. See \cite[\S4]{Faber_Hutz_Origin_2008} for details.


If we set $\gamma_0: Y \to \PP^1$ to be $\gamma_0((c,x,b)) = (c:1)$, then this map extends to a morphism $\gamma: X \to \PP^1$. For any fixed $c_0 \in k$, the equation $f^5_{c_0}(x) = b$ has degree~$32$ in the variable $x$. It follows that the divisor $\gamma^{-1}\left((c_0:1)\right)$ is relatively ample on all smooth fibers of $\pi$.

\begin{proof}[Proof of Theorem~\ref{Thm: Height Bound for c}]
	Note that the absolute logarithmic Weil height $h$ on $\PP^1$ is a Weil height associated to a divisor of degree~1. If we let $S = T = \PP^1$, and we fix $h$ as our height on $S, T$, then we are in a position to apply Version~2 of the Strong Mordell Conjecture to the family $\pi: X \to \Aff^1$. It follows that for any number field $k$, there exist positive constants $\alpha_1, \alpha_2$ such that 
	\[
		h\left(\gamma(P)\right) \leq \alpha_1 h(b) + \alpha_2 \quad \text{for all $b \in \Aff^1(k), P \in X_b(k)$.}
	\]
Note that we have used Remark~\ref{Rem: Singular Fibers} to take care of the singular fibers simultaneously.
If $P = P(x,c,b) \in Y_b(k)$ corresponds to a solution of $f_c^5(x) = b$, then $\gamma(P) = c$, and we recover
	\be
	\label{Eq: Height of c}
		h(c) \leq \alpha_1 h(b) + \alpha_2 \quad \text{for all $b \in \Aff^1(k), P \in X_b(k)$.}
	\ee

	To complete the proof, suppose that $x_0, b \in \Aff^1(k)$ and $c \in k$ satisfy $f_{c}^N(x_0) = b$ for some $N \geq 5$. Setting $x =  f_c^{N-5}(x_0)$, we find that $f_c^5(x) = b$. Now apply inequality~\eqref{Eq: Height of c} to deduce a height inequality for $c$.
\end{proof}


\section{The Canonical Height for the Morphism $f_c$}
\label{Sec: Corollary}
	
For each $c \in \alg{\QQ}$, one can define the canonical height associated to the morphism $f_c$ (of degree~2) by 
	\[
		\hat{h}_{f_c}(x) = \lim_{n \to \infty} \frac{h\left(f_c^n(x)\right)}{2^n} \qquad \text{for } x \in \Aff^1(\alg{\QQ}).
	\]
The article \cite{Call_Silverman_1993} shows that this limit exists, and that the functions $\hat{h}_{f_c}$ and $h$ differ by a bounded amount. In fact, they explain how the bound varies:

\begin{lem}[{\cite[Thm.~3.1]{Call_Silverman_1993}}]
\label{Lem: Can. Ht. Bnd}
	There exist constants $\beta_1, \beta_2 > 0$ such that for every $c \in \alg{\QQ}$ and $x \in \Aff^1(\alg{\QQ})$, 
		\[
			\left|\hat{h}_{f_c}(x) - h(x) \right| \leq \beta_1 h(c) + \beta_2.
		\]
\end{lem}

\begin{remark}
	With extra work, one can obtain the upper bound $ h(c) + \log 2$ in the lemma:
 cf. \cite[Thm.~B.2.5]{Hindry_Silverman_Book_2000}. 
\end{remark}

\begin{cor}
\label{Cor: x bound}
	Suppose $\beta_1, \beta_2$ are the constants from Lemma~\ref{Lem: Can. Ht. Bnd}. Let $N \geq 1$ be an integer, let $c \in \alg{\QQ}$ be an algebraic number, and let $x,b \in \Aff^1(\alg{\QQ})$ be points such that $f_c^N(x) = b$. Then
		\[
			\left|h(x) - \frac{1}{2^N}h(b) \right| \leq \left(1+\frac{1}{2^N}\right)\left(\beta_1 h(c) + \beta_2\right).
		\]
\end{cor}

\begin{proof}
	By the definition of the canonical height $\hat{h}_{f_c}$, we have $\hat{h}_{f_c}(f_c(x)) = 2 \hat{h}_{f_c}(x)$. Iterating this relation shows 
		$\hat{h}_{f_c}(b) = \hat{h}_{f_c}\left(f^N_c(x)\right) = 2^N \hat{h}_{f_c}(x)$, 
and hence
	\[
		\left|h(x) - \frac{1}{2^N}h(b) \right| = \left| \left(h(x) - \hat{h}_{f_c}(x) \right)
			+ \frac{1}{2^N}\left(\hat{h}_{f_c}(b) - h(b)\right) \right|.
	\]
Now apply the triangle inequality and the previous lemma.
\end{proof}

\begin{proof}[Proof of Corollary~\ref{Cor: Bound for pre-images}]
	Let $k$ be a number field, and let $\alpha_i, \beta_i$ be the absolute constants appearing in Theorem~\ref{Thm: Height Bound for c} and in Lemma~\ref{Lem: Can. Ht. Bnd}. Fix $b \in \Aff^1(k)$ and $c \in k$ for the duration of the proof.
	
	Suppose first that $x \in f_c^{-N}(b)(k)$ for some $N \geq 5$. Then Theorem~\ref{Thm: Height Bound for c} shows
	\be
	\label{Eq: c bound}
		h(c) \leq \alpha_1 h(b) + \alpha_2.
	\ee
Applying Corollary~\ref{Cor: x bound} gives
	\be
	\label{Eq: N at least 5}
		\ba
		h(x)  & \leq \frac{1}{2^N} h(b) + \left(1+ \frac{1}{2^N}\right)\left(\beta_1 h(c) + \beta_2 \right)  \\
			& \leq \frac{1}{32}h(b) + \frac{33}{32}\left(\beta_1 h(c) + \beta_2 \right) 
				\hspace{2cm} \text{since $N \geq 5$} \\
			& \leq \alpha_1' h(b) + \alpha_2' \hspace{4.28cm} \text{by \eqref{Eq: c bound}},
		\ea
	\ee
where $\alpha_1'$ and $\alpha_2'$ are constants depending only on $k$. 

	For $N \leq 4$, since $f_c(x) = x^2 + c$ has degree~2, we find that
	\be
	\label{Eq: Small N}
			\# \Bigg\{ \bigcup_{1 \leq N \leq 4} f^{-N}_c(b)(k) \Bigg\} 
				 \leq \sum_{1 \leq N \leq 4} \# f^{-N}_c(b)(k) 
				 \ \leq \sum_{1 \leq N \leq 4} 2^N = 30.
	\ee
	
	Combining~\eqref{Eq: N at least 5} and~\eqref{Eq: Small N} shows
	\benn
		\ba
			\# \Bigg\{ \bigcup_{N \geq 1} f^{-N}_c(b)(k) \Bigg\} &\leq 30 +
			\#\left\{x \in \Aff^1(k): h(x) \leq \alpha_1' h(b) + \alpha_2' \right\}. \\
		\ea
	\eenn
To complete the proof, we note that the number of points in $\Aff^1(k)$ with absolute logarithmic height at most~$t$ is bounded by $O_k(\exp(\delta t))$ for some positive constant $\delta$ depending on $k$.  See \cite[Thm.~B.6.2]{Hindry_Silverman_Book_2000}.
\end{proof}

\bibliographystyle{plain}
\bibliography{xander_bib}

\end{document}